\documentclass{amsart}

\usepackage[cmtip,all]{xy}

\newtheorem{theorem}{Theorem}[section]
\newtheorem{proposition}[theorem]{Proposition}
\newtheorem{lemma}[theorem]{Lemma}
\newtheorem{corollary}[theorem]{Corollary}

\theoremstyle{definition}
\newtheorem{definition}[theorem]{Definition}
\newtheorem{example}[theorem]{Example}

\theoremstyle{remark}
\newtheorem{remark}[theorem]{Remark}

\numberwithin{equation}{section}

\begin{document}

\title[Gaussian Processes with Sample Paths in RKBS]{Gaussian Processes with Sample Paths in Reproducing Kernel Banach Spaces}


\author[T. Karvonen]{Toni Karvonen}
\address{School of Engineering Sciences, Lappeenranta-Lahti University of Technology LUT, Yliopistonkatu 34, 53850 Lappeenranta, Finland}
\curraddr{}
\email{toni.karvonen@lut.fi}
\thanks{The first author was supported by the Research Council of Finland grants 359183 (``Flagship of Advanced Mathematics for Sensing, Imaging and Modelling'') and 368086 (``Inference and approximation under misspecification'').}

\author[R. K. H. Sørensen]{Rasmus K. H. Sørensen}
\address{Department of Applied Mathematics and Computer Science, Technical University of Denmark, Richard Petersens Plads Building 324, 2800 Kongens Lyngby, Denmark}
\curraddr{}
\email{rkhso@dtu.dk}
\thanks{The second author was supported by Danish Data Science Academy, which is funded by the Novo Nordisk Foundation (NNF21SA0069429) and VILLUM FONDEN (40516).}



\date{}

\dedicatory{}

\begin{abstract}
    We investigate the connection between Gaussian processes and Gaussian random elements in reproducing kernel Banach spaces.
    We show that the covariance operator of a weak second-order Radon probability measure on such a space is uniquely determined by a symmetric positive-definite function.
    In the Gaussian case, we characterize those functions that arise from covariance operators in terms of \( \gamma \)-radonifying operators.
    Building on these results, we extend Driscoll's theorem to the reproducing kernel Banach space setting.
\end{abstract}

\maketitle

\section{Introduction}
\label{sec:introduction}

Let \( (\Omega,\mathcal{A},\mu) \) be a probability space and \( \xi : T \times \Omega \to \mathbb{R} \) a centered Gaussian process over a non-empty set \( T \) with covariance function \( K \).
For such a random process, questions concerning the regularity of its sample paths arise immediately.
A classical result in this direction asserts that the sample paths of \( \xi \) belong to the reproducing kernel Hilbert space \( H(K) \) of \( K \) with probability zero whenever \( H(K) \) is infinite-dimensional.
This result was originally stated without proof by Parzen \cite[eq.~34]{parzen1963}, and later proved by Kallianpur \cite{kallianpur1970a} and Le Page \cite{lepage1973}.

This motivated Driscoll \cite{driscoll1973} to ask whether there exists another reproducing kernel Hilbert space \( H(L) \) that contains the sample paths of \( \xi \) almost surely.
Under appropriate assumptions, Driscoll proved the zero--one law
\begin{equation} \label{eq:Driscoll:zero:one}
    \mu(\xi \in H(L)) = 0
    \qquad \text{or} \qquad
    \mu(\xi \in H(L)) = 1,
\end{equation}
and established a necessary and sufficient condition for \( \mu(\xi \in H(L)) = 1 \) to hold.
The condition given by Driscoll is equivalent to requiring that \( H(K) \subseteq H(L) \) and that the linear operator \( R : H(L) \to H(L) \), defined by \( R L(\cdot,t) = K(\cdot,t) \) for \( t \in T \), is nuclear.
Fortet \cite{fortet1974} subsequently introduced the terminology that \( L \) \emph{dominates} \( K \) if \( H(K) \subseteq H(L) \), and that \( L \) \emph{\( n \)-dominates} \( K \) if, in addition, \( R \) is nuclear.
In modern terminology, the latter is commonly referred to as \emph{nuclear dominance}.

The assumptions made by Driscoll in the original paper are that \( T \) is a separable metric space and that the sample paths of \( \xi \) are continuous with respect to the metric topology on \( T \).
Driscoll describes the latter condition as ``a rather restrictive assumption, even for Gaussian processes''~\cite[p.~309]{driscoll1973}.
Nevertheless, this assumption allows \( \xi \) to be regarded as a random element taking values in the space \( C(T) \) of continuous functions on \( T \), thereby addressing a fundamental measurability issue that arises for random processes and the cylindrical \( \sigma \)-algebra on \( \mathbb{R}^T \).

Lukić and Beder~\cite{lukic2001} later demonstrated that neither of Driscoll's additional assumptions is necessary.
The crux of their analysis is the canonical metric on \( T \) induced by the kernel \( L \), with respect to which every function in \( H(L) \) is continuous.
Under the nuclear dominance assumption, \( H(L) \) may be assumed to be separable, and hence \( T \) admits a countable dense subset with respect to this metric.
Random elements in \( H(L) \) can therefore be defined by specifying them on such a subset and extending them uniquely by continuity, which yields a natural candidate for a modification of \( \xi \) taking values in \( H(L) \).
In this way the measurability difficulties that motivated Driscoll's additional assumptions are resolved by reducing the problem to a countable dense subset of \( T \), where they do not arise.

In this paper, we extend the analysis of Lukić and Beder~\cite{lukic2001} to the Banach space setting and generalize Driscoll's theorem to separable reproducing kernel Banach spaces.
Our main results are Theorems~\ref{thm:Driscoll:one} and~\ref{thm:Driscoll:zero}, which establish the ``one'' and ``zero'' parts of Driscoll's theorem, respectively.
These results are summarized in the following theorem using \( \gamma \)-radonifying operators~\cite{hytonen2017,neerven2010}.
The \( \gamma \)-radonifying operators can be thought of as operators that map ``white noise'' in a Hilbert space into a well-defined Gaussian random element in a Banach space (see Theorem~\ref{thm:gamma:radonifying}).

\begin{theorem} \label{thm:Driscoll}
    Let \( (\Omega,\mathcal{A},\mu) \) be a complete probability space, let \( X \subseteq \mathbb{R}^T \) be a separable reproducing kernel Banach space, and let \( \xi : T \times \Omega \to \mathbb{R} \) be a centered Gaussian process with covariance function \( K \) and associated Hilbert space \( H(K) \).
    \begin{enumerate}
        \item If \( H(K) \subseteq X \) and the embedding \( \iota : H(K) \to X \) is \( \gamma \)-radonifying, then there exists a centered Gaussian random element \( \eta : \Omega \to X \) such that \( \xi_t = \eta_t \) almost surely for all \( t \in T \).

        \item If \( H(K) \not\subseteq X \) or if the embedding \( \iota : H(K) \to X \) is not \( \gamma \)-radonifying, then
        \begin{equation*}
            \mu(\{ \omega \in \Omega : \xi(\cdot,\omega) \in X \}) = 0.
        \end{equation*}
    \end{enumerate}
\end{theorem}

\begin{remark} \label{rem:measurability}
    Due to subtle measurability issues discussed in Section~\ref{sec:Driscoll:theorem} we purposely avoid statements such as ``\( \xi \) has sample paths almost surely in \( X \)''.
\end{remark}

As in Lukić and Beder~\cite{lukic2001}, a cornerstone of our analysis is the connection between random elements in reproducing kernel Banach spaces and the random processes they define.
We show that the covariance operator of a weak second-order Radon probability measure on such a space is uniquely determined by a symmetric positive-definite function.
In particular, we characterize those functions that arise as covariance functions of Radon Gaussian measures.
This result replaces the Prokhorov theorem used in Driscoll's original paper~\cite[Theorem~2]{driscoll1973}.
Moreover, we promote statements on sample path regularity from random processes to random elements, whereas classical results provide measurability only with respect to the cylindrical \( \sigma \)-algebra generated by the point evaluations.
This promotion permits the consideration of pushforwards by measurable mappings defined on the Banach space rather than on the space of real-valued functions.

Several authors have established \( \gamma \)-radonifying embeddings of reproducing kernel Hilbert spaces associated with Gaussian processes into classical function spaces.
A fundamental result due to Brzeźniak \cite[Theorem~3.6]{brzezniak1996} shows that the indefinite integration operator is \( \gamma \)-radonifying as an operator from \( L^2(0,1) \) into the fractional Sobolev space \( W^{\alpha,p}(0,1) \) for \( p \in (2,\infty) \) and \( \alpha \in \left(1/p, 1/2\right) \).
This yields fractional Sobolev regularity, and hence Hölder regularity, of Brownian motion.
The results in \cite{brzezniak1996} are formulated in terms of almost sure convergence of Gaussian series rather than in the language of \( \gamma \)-radonifying operators.
Further details, refinements, and results on Besov regularity can be found in \cite{brzezniak1996} and the references therein.

\section{Preliminaries}
\label{sec:preliminaries}

Here we recall some basic definitions and properties of random elements in locally convex spaces and their distributions.
Further details and proofs can be found in the classical texts \cite{bogachev1998,vakhania1987}.
Throughout this section, \( X \) denotes a locally convex space, \( X^* \) its topological dual, and \( (\Omega,\mathcal{A}) \) a measurable space.
We write \( \langle \cdot, \cdot \rangle \) for the dual pairing between \( X \) and \( X^* \) defined by \( \langle x ,f \rangle = f(x) \) for \( x \in X \) and \( f \in X^* \).

The \( \sigma \)-algebra generated by a family \( \Gamma \subseteq X^* \) is called the \emph{cylindrical \( \sigma \)-algebra with respect to \( \Gamma \)} and will be denoted by \( \mathcal{C}(X,\Gamma) \).
It is the smallest \( \sigma \)-algebra on \( X \) with respect to which every \( f \in \Gamma \) is measurable.
In the special case \( \Gamma = X^* \), we write \( \mathcal{C}(X) \).
The Borel \( \sigma \)-algebra on \( X \) will be denoted by \( \mathcal{B}(X) \).
In some important cases \( \mathcal{B}(X) \) is strictly larger than \( \mathcal{C}(X,\Gamma) \).
However, if \( X \) is a separable Fréchet space and \( \Gamma \) separates points, the \( \sigma \)-algebras coincide.

An essential analytical tool in the study of probability measures on locally convex spaces is the characteristic functional, also called the Fourier transform.

\begin{definition} \label{def:characteristic:functional}
    Let \( \mu \) be a probability measure on \( \mathcal{C}(X,\Gamma) \).
    The \emph{characteristic functional} of \( \mu \) is the functional
    \begin{equation} \label{eq:characteristic:functional}
        \widehat{\mu}:\Gamma\to\mathbb{C}, \qquad
        f \mapsto \int_X e^{if(x)} \, \mu(dx).
    \end{equation}
\end{definition}

By the Lebesgue dominated convergence theorem, \( \widehat{\mu} \) is sequentially continuous on \( \Gamma \) endowed with the weak-\( * \) topology.
In general, however, continuity of \( \widehat{\mu} \) requires a stronger topology on \( \Gamma \), such as the topology of uniform convergence on compact subsets of \( X \), or the Mackey topology if \( X \) is quasi-complete.

An important property of the characteristic functional is that if two probability measures \( \mu_1 \) and \( \mu_2 \) on \( \mathcal{C}(X,\Gamma) \) satisfy \( \widehat{\mu}_1 = \widehat{\mu}_2 \), then \( \mu_1 = \mu_2 \).

\begin{definition} \label{def:gaussian:measure}
    A probability measure \( \gamma \) on \( (X,\mathcal{C}(X)) \) is called a \emph{Gaussian measure} if the characteristic functional of \( \gamma \) has the form
    \begin{equation} \label{eq:gaussian:characteristic:functional}
        \widehat{\gamma}(f) =
        \exp\left(ia(f)-\frac{1}{2}R(f,f)\right),
        \qquad \forall f\in X^*,
    \end{equation}
    where \( a : X^* \to \mathbb{R} \) is a linear functional and \( R : X^* \times X^* \to \mathbb{R} \) is a symmetric positive bilinear form.
\end{definition}

It is well known that a probability measure \( \gamma \) on \( (X,\mathcal{C}(X)) \) is Gaussian if and only if, for every \( f \in X^* \), the pushforward \( \gamma \circ f^{-1} \) is a Gaussian measure on \( (\mathbb{R}, \mathcal{B}(\mathbb{R})) \).
Indeed, many authors take this property as the definition of a Gaussian measure.

\begin{definition} \label{def:random:element}
    Let \( \mu \) be a probability measure on \( (\Omega,\mathcal{A}) \).
    A \( (\mathcal{A}, \mathcal{C}(X)) \)-measurable map \( \xi : \Omega \to X \) is called a \emph{random element in \( X \)}.
    It is called \emph{Gaussian} if the pushforward measure \( \gamma = \mu \circ \xi^{-1} \) is Gaussian on \( (X,\mathcal{C}(X)) \).
\end{definition}

\begin{remark}
    Some authors require a random element \( \xi : \Omega \to X \) to be \emph{strongly measurable}, meaning that \( \xi \) is the pointwise limit of a sequence of simple functions.
    This distinction becomes important when considering vector-valued integrals.
    For separable Banach spaces, the Pettis measurability theorem \cite{pettis1938} ensures that the two definitions coincide.
    For further details on vector-valued integration, see \cite{diestel1977}.
\end{remark}

For \( 1 \leq p < \infty \), let \( L^p(\mu) \) denote the Lebesgue space of \( p \)-integrable functions with respect to the measure \( \mu \).
A probability measure \( \mu \) is said to be of \emph{weak \( p \)-order} if \( X^* \subseteq L^p(\mu) \).
For probability measures of weak order \( p \geq 2 \), the mean and covariance can be defined as follows.

\begin{definition} \label{def:mean}
    Let \( \mu \) be a measure on \( (X,\mathcal{C}(X)) \) such that \( X^*\subseteq L^1(\mu) \).
    The \emph{mean} of \( \mu \) is the linear functional
    \begin{equation} \label{eq:mean}
        a_\mu:X^*\to\mathbb{R},\qquad f\mapsto\int_X f(x)\mu(dx).
    \end{equation}
\end{definition}

\begin{definition} \label{def:covariance}
    Let \( \mu \) be a measure on \( (X, \mathcal{C}(X)) \) such that \( X^* \subseteq L^2(\mu) \).
    The \emph{covariance} of \( \mu \) is the bilinear form
    \begin{equation} \label{eq:covariance}
        R_\mu : X^* \times X^* \to \mathbb{R}, \qquad
        (f,g) \mapsto \int_X (f(x) - a_\mu(f))(g(x) - a_\mu(g)) \, \mu(dx).
    \end{equation}
\end{definition}

It follows from Fernique's theorem that every Gaussian measure is of weak order \( p \) for all \( p \in [1,\infty) \).
For a Radon Gaussian measure, the mean can be represented by an element of \( X \) and the covariance by a continuous linear operator \( R : X^* \to X \).
We recall that a measure \( \mu \) on \( (X, \mathcal{B}(X)) \) is called a \emph{Radon measure} if, for every \( B \in \mathcal{B}(X) \) and \( \epsilon > 0 \), there exists a compact set \( K \subseteq B \) such that \( \mu(B \setminus K) < \epsilon \).

\begin{definition} \label{def:Radon:Gaussian:measure}
    A Radon measure \( \gamma \) on \( (X, \mathcal{B}(X)) \) is called a \emph{Radon Gaussian measure} if the restriction of \( \gamma \) to \( (X, \mathcal{C}(X)) \) is a Gaussian measure.
\end{definition}

\begin{theorem} \label{thm:Gaussian:moments}
    Let \( \gamma \) be a Radon Gaussian measure on \( (X,\mathcal{B}(X)) \).
    Then
    \begin{enumerate}
        \item There exists a unique element \( a \in X \) such that
        \begin{equation*}
            f(a) = \int_X f(x) \, \gamma(dx), \qquad \forall f \in X^*.
        \end{equation*}

        \item There exists a unique linear operator \( R : X^* \to X \) such that
        \begin{equation*}
            \langle Rf,g \rangle =
            \int_X (f(x)-a_\gamma(f))(g(x)-a_\gamma(g)) \, \gamma(dx)
            \qquad \forall f, g \in X^*.
        \end{equation*}
    \end{enumerate}
\end{theorem}

\begin{remark}
    It follows from Ulam's theorem that a Borel probability measure \( \mu \) on a complete metric space \( X \) is a Radon measure if and only if there exists a separable subset \( Y \subseteq X \) such that \( \mu(Y) = 1 \).
\end{remark}

The concept of a Gaussian measure is closely related to that of Gaussian random processes.
Let \( T \) be a non-empty set, and let \( \mathbb{R}^T \) be the space of real-valued functions on \( T \), endowed with the topology of pointwise convergence.

\begin{definition} \label{def:random:process}
    A \emph{random process} is a function \( \xi: T \times \Omega \to \mathbb{R} \) such that, for every \( t \in T \), the function \( \xi_t: \Omega \to \mathbb{R},\; \omega\mapsto \xi(t, \omega) \) is \( \mathcal{A} \)-measurable.
\end{definition}

A straightforward calculation shows that \( \xi : T \times \Omega \to \mathbb{R} \) is a random process if and only if the map \( \omega \mapsto \xi(\cdot,\omega) \) is \( (\mathcal{A}, \mathcal{C}(\mathbb{R}^T)) \)-measurable.
Hence a random process can equivalently be regarded as a random element in \( (\mathbb{R}^T, \mathcal{C}(\mathbb{R}^T)) \).
For \( \omega \in \Omega \), the realization \( \xi(\cdot,\omega) \) is called a \emph{sample path}, and \( \mathbb{R}^T \) is called the \emph{path space}.

\begin{definition} \label{def:gaussian:process}
    Let \( \mu \) be a probability measure on \( (\Omega,\mathcal{A}) \).
    A random process \( \xi: T \times \Omega \to \mathbb{R} \) is called a \emph{Gaussian process} if the induced measure \( \mu^\xi \) on \( (\mathbb{R}^T,\mathcal{C}(\mathbb{R}^T)) \) is Gaussian.
\end{definition}

For \( t \in T \), let \( \delta_t:\mathbb{R}^T\to\mathbb{R}, \, x \mapsto x(t) \) be the evaluation functional at \( t \).
Using the characterization \( (\mathbb{R}^T)^* = \operatorname{span}\{ \delta_t : t \in T \} \), it follows that a random process \( \xi: T \times \Omega \to \mathbb{R} \) is Gaussian if and only if \( (\xi_t)_{t \in T_0} \) is a Gaussian random vector in \( \mathbb{R}^{T_0} \) for every finite subset \( T_0 \subseteq T \).

\begin{proposition} \label{prop:Gaussian:process:moments}
    Let \( \gamma \) be a Gaussian measure on \( (\mathbb{R}^T,\mathcal{C}(\mathbb{R}^T)) \).
    \begin{enumerate}
        \item There exists a unique function \( a : T \to \mathbb{R} \) such that
        \begin{equation*}
            a_\gamma(f) = f(a), \qquad \forall f \in (\mathbb{R}^T)^*.
        \end{equation*}
        \item There exists a unique function \( K : T \times T \to \mathbb{R} \) such that
        \begin{equation*}
            R_\gamma(\delta_s,\delta_t) = K(s,t), \qquad \forall s,t \in T.
        \end{equation*}
    \end{enumerate}
\end{proposition}

Because \( (\mathbb{R}^T)^* = \operatorname{span}\{ \delta_t : t \in T \} \) and \( R_\gamma \) is bilinear, it follows that \( R_\gamma \) is uniquely determined by the function \( K \) in Proposition~\ref{prop:Gaussian:process:moments}.
Accordingly, \( K \) is called the \emph{covariance function} of \( \gamma \).
It is readily verified that \( K \) is symmetric and positive-definite in the following sense.

\begin{definition} \label{def:symmetric:positive:function}
    Let \( K : T \times T \to \mathbb{R} \) be a function.
    \begin{enumerate}
        \item \( K \) is \emph{symmetric} if \( K(s,t) = K(t,s) \) for all \( s,t \in T \).
        \item \( K \) is \emph{positive-definite} if for every finite subset \( T_0 \subseteq T \) and every collection of scalars \( (c_t)_{t \in T_0} \subset \mathbb{R} \), it holds that
        \begin{equation*}
        \sum_{s,t \in T_0} c_s c_t K(s, t) \geq 0.
        \end{equation*}
    \end{enumerate}
\end{definition}

A remarkable property of Gaussian processes is that every symmetric positive-definite function determines a unique (centered) Gaussian measure on \( (\mathbb{R}^T,\mathcal{C}(\mathbb{R}^T)) \).
This follows immediately from the celebrated Kolmogorov extension theorem.
For Gaussian measures on a general locally convex space \( X \), however, describing the bilinear forms that arise as Gaussian covariances is much more delicate.

\section{Covariance Operators}
\label{sec:covariance:operators}

Throughout this section, \( X \) denotes a Banach space with norm \( \|\cdot\| \) and \( H \) denotes a Hilbert space with inner product \( (\cdot,\cdot) \).
Their dual spaces \( X^* \) and \( H^* \) are endowed with the respective strong topologies unless otherwise specified.
We write \( \mathcal{L}(H,X) \) for the space of continuous linear operators from \( H \) to \( X \), endowed with the operator-norm topology.
When there is a risk of confusion, we add a subscript to specify the underlying space; for example, \( \|\cdot\|_{X} \) denotes the norm on \( X \).

\begin{definition} \label{def:symmetric:positive}
    Let \( R : X^* \to X \) be an operator.
    \begin{enumerate}
        \item \( R \) is \emph{symmetric} if \( \langle Rf, g \rangle = \langle Rg, f \rangle \) for all \( f,g \in X^* \).
        \item \( R \) is \emph{positive} if \( \langle Rf, f \rangle \geq 0 \) for all \( f \in X^* \).
    \end{enumerate}
\end{definition}

It is readily verified that every covariance operator is symmetric and positive.
Therefore we begin by recalling some of their basic properties.
Proofs of the results cited can be found in \cite[Chapter~III]{vakhania1987}.

\begin{proposition} \label{prop:symmetric}
    Let \( R : X^* \to X \) be a symmetric positive operator.
    Then \( R \) is linear and continuous.
\end{proposition}

Since every symmetric positive operator is automatically continuous and linear, we will omit these qualifiers when referring to such operators.

\begin{lemma} \label{lem:cauchy:schwarz}
    Let \( R : X^* \to X \) be a symmetric positive operator.
    Then \( R \) satisfies the \emph{Cauchy--Schwarz inequality}
    \begin{equation} \label{eq:cauchy:schwarz}
        |\langle Rf, g \rangle| \leq \sqrt{\langle Rf,f \rangle}\sqrt{\langle Rg,g \rangle},\qquad \forall f,g \in X^*.
    \end{equation}
\end{lemma}

From the Cauchy--Schwarz inequality it follows that every symmetric positive operator \( R : X^* \to X \) naturally induces an inner product on its range \( R(X^*) \) with
\begin{equation} \label{eq:cameron:martin:inner}
    (Rf, Rg) = \langle Rf, g \rangle, \qquad f,g \in X^*.
\end{equation}
In general, \( R(X^*) \) is not complete with respect to the norm induced by the inner product.
However, since \( X \) is complete, one obtains a Hilbert space completion of \( R(X^*) \) as a subspace of \( X \).

\begin{proposition} \label{prop:cameron:martin}
    Let \( R : X^* \to X \) be a symmetric positive operator.
    Then there exists a unique Hilbert space \( H \subseteq X \) such that
    \begin{enumerate}
        \item The embedding \( \iota: H \to X \) is continuous;
        \item \( R(X^*) \) is a dense subspace of \( H \);
        \item \( (Rf,Rg) = \langle Rf, g \rangle \) for all \( f,g\in X^* \).
    \end{enumerate}
\end{proposition}

\begin{remark} \label{rem:cameron:martin}
    For every Radon Gaussian measure \( \gamma \) on \( (X,\mathcal{B}(X)) \) with covariance operator \( R : X^* \to X \), the Hilbert space \( H \) coincides with the Cameron--Martin space of \( \gamma \).
    For further details on Cameron--Martin spaces, see \cite{bogachev1998}.
\end{remark}

The Hilbert space \( H \) in Proposition~\ref{prop:cameron:martin} is called the \emph{Hilbert space associated with \( R \)} and will be denoted by \( H(R) \).
Because \( R(X^*) \) is a dense subspace of \( H(R) \) and the embedding \( \iota : H(R) \to X \) is continuous, it follows from \eqref{eq:cameron:martin:inner} that the inner product on \( H(R) \) satisfies the \emph{reproducing property}
\begin{equation} \label{eq:reproducing:property}
    (x, Rf) = f(x), \qquad x \in H(R),\ f \in X^*.
\end{equation}

The following factorization lemma shows that \( R \) is uniquely determined by \( H(R) \).
For an operator \( A \in \mathcal{L}(H,X) \) we denote by \( A^{*} \) its adjoint, and we identify \( H \) with its dual \( H^{*} \) via the Riesz representation theorem.

\begin{lemma}[Factorization] \label{lem:factorization}
    Let \( R : X^* \to X \) be a symmetric positive operator.
    Then \( R = \iota \iota^* \) where \( \iota : H(R) \to X \) is the embedding.
\end{lemma}
\begin{proof}
    By the reproducing property \eqref{eq:reproducing:property}, it follows that
    \begin{equation*}
        (x, Rf) = \langle \iota x, f \rangle_X = \langle x, \iota^* f \rangle_H, \qquad \forall x \in H(R), f \in X^*.
    \end{equation*}
    Because this holds for all \( x \in H(R) \), the Riesz representation theorem implies that \( \iota^* f \in H(R)^* \) can be identified with \( Rf \) for all \( f \in X^* \).
    Therefore, \( R = \iota \iota^* \).
\end{proof}

It is clear that the converse of Lemma~\ref{lem:factorization} also holds:
if \( H \) is a Hilbert space and \( A \in \mathcal{L}(H,X) \), then \( R = A A^{*} \) is a symmetric positive operator.
In this case, \( H(R) = A(H) \) and \( H(R) \) is isometrically isomorphic to \( H / \ker(A) \).

\begin{corollary} \label{cor:spectral:decomposition}
    Let \( R : X^* \to X \) be a symmetric positive operator and suppose that \( H(R) \) is separable with an orthonormal basis \( (h_k) \).
    Then
    \begin{equation} \label{eq:spectral:decomposition}
        R f = \sum_{k=1}^\infty \langle h_k , f \rangle h_k, \qquad \forall f \in X^*.
    \end{equation}
\end{corollary}

For a Radon Gaussian measure \( \gamma \) on \( (X,\mathcal{B}(X)) \) with covariance operator \( R \), it is well known that \( H(R) \) is separable.
Hence \( R \) admits the representation \eqref{eq:spectral:decomposition}.
This representation naturally suggests decomposing \( \gamma \) as a series expansion in the orthonormal basis \( (h_k) \), with independent standard Gaussian coefficients.

Throughout this section, let \( (\Omega, \mathcal{A}, \mu) \) be a fixed probability space supporting a countable sequence \( (g_k) \) of independent standard Gaussian random variables.
Such a sequence is called a \emph{Gaussian sequence}.

\begin{theorem} \label{thm:gaussian:series}
    Let \( \gamma \) be a centered Radon Gaussian measure on \( (X,\mathcal{B}(X)) \) with covariance operator \( R \) and \( (h_k) \) an orthonormal basis for \( H(R) \).
    Then the series
    \begin{equation} \label{eq:gaussian:series}
        \sum_{k=1}^{\infty} g_k h_k
    \end{equation}
    converges almost surely in \( X \).
    The distribution of its sum is \( \gamma \).
\end{theorem}
\begin{proof}
    For \( n \in \mathbb{N} \), let \( S_n = \sum_{k=1}^{n} g_k h_k \) and \( R_n = \sum_{k=1}^{n} \langle h_k, \cdot \rangle h_k \) be the \( n \)th partial sums.
    Then \( S_n \) is a separably valued random element in \( X \) with distribution \( \gamma_n \).
    Independence of \( (g_k) \) gives
    \begin{equation*}
        \widehat{\gamma}_n(f)
        = \prod_{k=1}^{n}\exp{\left( -\frac{1}{2}\langle h_k, f \rangle^2 \right)}
        = \exp{\left( -\frac{1}{2}\langle R_nf, f \rangle \right)}
        \qquad \forall f \in X^*.
    \end{equation*}
    Hence \( S_n \) is a Gaussian random element.
    Because \( R_n \to R \) in the strong operator topology on \( \mathcal{L}(X^*,X) \) by the series representation \eqref{eq:spectral:decomposition}, we obtain
    \begin{equation*}
        \lim_{n\to\infty}\widehat{\gamma}_n(f) = \exp\left(-\frac{1}{2}\langle R f, f \rangle \right) = \widehat{\gamma}(f), \qquad \forall f \in X^*.
    \end{equation*}
    From the Itô--Nisio theorem (see \cite[Theorem~V.2.4]{vakhania1987}) it follows that \( S_n \) converges almost surely in \( X \).
    The limit \( S = \lim_{n\to\infty}S_n \) has distribution \( \gamma \).
\end{proof}

\begin{remark}
    For Gaussian series, almost sure convergence and convergence in the Lebesgue--Bochner space \( L^{2}(\Omega;X) \) coincide (see, for example, \cite[Corollary~6.4.4]{hytonen2017}).
    In what follows, we formulate results in terms of almost sure convergence, with equivalent formulations available in \( L^{2}(\Omega;X) \).
\end{remark}

A series of the form \eqref{eq:gaussian:series} is referred to as a \emph{Gaussian series} and is treated in detail in \cite{bogachev1998,hytonen2017,vakhania1987}.
Gaussian series are important for characterizing covariance operators of Radon Gaussian measures.
We present one such characterization in terms of the associated Hilbert space, formulated using \( \gamma \)-radonifying operators.
A detailed treatment of \( \gamma \)-radonifying operators can be found in \cite{hytonen2017,neerven2010}.

For the next definition, we introduce the notation \( \mathcal{F}(H) \) for the collection of all finite orthonormal systems in \( H \).

\begin{definition} \label{def:gamma:summing}
    A linear operator \( A : H \to X \) is called \emph{\( \gamma \)-summing} if
    \begin{equation} \label{eq:gamma:summing}
        \sup_{(h_k)_{k=1}^n \in \mathcal{F}(H)} \int_\Omega \left\| \sum_{k=1}^n g_k(\omega) A h_k \right\|^2 \mu(d\omega) < \infty.
    \end{equation}
\end{definition}

\begin{remark}
    By the Kahane--Khintchine inequality, the exponent \( p = 2 \) in \eqref{eq:gamma:summing} may be replaced by any \( p \in [1,\infty) \).
    The choice \( p = 2 \) is standard in the literature.
\end{remark}

The collection of \( \gamma \)-summing operators will be denoted by \( \gamma_\infty(H,X) \).
It is readily verified that every \( A \in \gamma_\infty(H,X) \) is bounded by considering orthonormal systems consisting of a single element.
Moreover, \( \gamma_\infty(H,X) \) can be equipped with a norm for which it becomes a Banach space that is continuously embedded into \( \mathcal{L}(H,X) \).

\begin{lemma} \label{lem:gamma:summing}
    The space \( \gamma_\infty(H,X) \) is a Banach space with respect to the norm
    \begin{equation} \label{eq:gamma:summing:norm}
        \| A \|_{\gamma_\infty(H,X)} = \sup_{(h_k)_{k=1}^n \in \mathcal{F}(H)} \left( \int_\Omega \left\| \sum_{k=1}^n g_k(\omega) A h_k \right\|^2 \mu(d\omega) \right)^{1/2}.
    \end{equation}
\end{lemma}

For finite-rank operators the \( \gamma_\infty \)-norm admits an explicit expression.
Recall that an \( A \in \mathcal{L}(H,X) \) has \emph{finite rank} if its range \( A(H) \) is finite-dimensional, or equivalently, if there exists an orthonormal system \( (h_k)_{k=1}^n \subset H \) such that
\begin{equation} \label{eq:finite:rank}
    A h = \sum_{k=1}^n (h, h_k)\, A h_k, \qquad \forall h \in H.
\end{equation}

\begin{lemma} \label{lem:finite:rank}
    Suppose that \( A \in \mathcal{L}(H,X) \) is a finite-rank operator and let \( (h_k)_{k=1}^n \) be an orthonormal system in \( H \) such that \eqref{eq:finite:rank} holds.
    Then \( A \in \gamma_\infty(H,X) \) and
    \begin{equation} \label{eq:finite:rank:gamma:norm}
        \| A \|_{\gamma_\infty(H,X)}^2 = \int_\Omega \left\| \sum_{k=1}^n g_k(\omega) A h_k \right\|^2 \mu(d\omega).
    \end{equation}
\end{lemma}

The closure of the subspace of finite rank operators in \( \gamma_\infty(H,X) \) will be denoted by \( \gamma(H,X) \), and operators in \( \gamma(H,X) \) are called \emph{\( \gamma \)-radonifying}.
The two spaces coincide when \( X \) does not contain a closed subspace isomorphic to \( c_0 \) (the Banach space of null sequences in \( \mathbb{R} \)).
In general, however, \( \gamma(H,X) \) may be a proper subspace.

It is immediate that \( \gamma \)-radonifying operators can be approximated in \( \gamma(H,X) \) by finite rank operators.
Since \( \gamma(H,X) \) is continuously embedded into \( \mathcal{L}(H,X) \), it follows that every \( \gamma \)-radonifying operator is compact and \( \ker(A)^\perp \) is separable.
Moreover, \( \gamma \)-radonifying operators admit the following characterization.

\begin{theorem} \label{thm:gamma:radonifying}
    Suppose that \( H \) is separable with an orthonormal basis \( (h_k) \).
    Then an operator \( A \in \mathcal{L}(H,X) \) is \( \gamma \)-radonifying if and only if the series
    \begin{equation} \label{eq:gamma:radonifying}
        \sum_{k=1}^\infty g_k A h_k
    \end{equation}
    converges almost surely in \( X \).
\end{theorem}

It is straightforward to verify that, in the special case where \( X \) is a Hilbert space and \( A \in \mathcal{L}(H,X) \), the Gaussian series \eqref{eq:gamma:radonifying} converges almost surely if and only if \( A \) is a Hilbert--Schmidt operator.
Recall that \( A \) is Hilbert--Schmidt if
\begin{equation} \label{eq:hilbert:schmidt}
    \sup_{(h_k)_{k=1}^n \in \mathcal{F}(H)} \sum_{k=1}^n \| Ah_k \|^2 < \infty.
\end{equation}

Theorems~\ref{thm:gaussian:series} and \ref{thm:gamma:radonifying} together characterize the covariance operators of Radon Gaussian measures.
See also \cite[Proposition~8.6]{neerven2010} and \cite[Theorem~V.5.4]{vakhania1987}.

\begin{theorem} \label{thm:gaussian:covariance}
    A symmetric positive operator \( R : X^* \to X \) is the covariance operator of a centered Radon Gaussian measure on \( (X,\mathcal{B}(X)) \) if and only if the embedding \( \iota : H(R) \to X \) is \( \gamma \)-radonifying.
\end{theorem}
\begin{proof}
    Suppose that there exists a Radon Gaussian measure \( \gamma \) on \( (X,\mathcal{B}(X)) \) with covariance operator \( R \).
    Then \( H(R) \) is separable with orthonormal basis \( (h_k) \), and by Theorem~\ref{thm:gaussian:series} the Gaussian series \( \sum_{k=1}^\infty g_k h_k \) converges almost surely in \( X \).
    Hence, by Theorem~\ref{thm:gamma:radonifying}, the embedding \( \iota \) is \( \gamma \)-radonifying.

    Conversely, suppose that \( \iota \) is \( \gamma \)-radonifying.
    Because \( \iota \) is injective, it follows that \( H(R) = \ker(\iota)^\perp \), and therefore \( H(R) \) is separable since \( \ker(\iota)^\perp \) is separable for \( \gamma \)-radonifying operators.
    Let \( (h_k) \) be an orthonormal basis for \( H(R) \).
    Then the Gaussian series \( \sum g_k h_k \) converges almost surely in \( X \) by Theorem~\ref{thm:gamma:radonifying}.
    Let \( S = \sum g_k h_k \) be the limit and let \( \gamma \) be its distribution.
    From the Lebesgue dominated convergence theorem it follows that
    \begin{equation*}
        \widehat{\gamma}(f) = \exp{\left( -\frac{1}{2}\langle Rf, f \rangle \right)}, \quad \forall f \in X^*.
    \end{equation*}
    Hence \( \gamma \) is a centered Gaussian measure.
    It is Radon because \( S \) is separably valued, and its covariance operator is \( R \).
\end{proof}

\section{Reproducing Kernel Banach Spaces}
\label{sec:rkbs}

Natural candidates for sample path spaces of random processes are function spaces that are continuously embedded into \( \mathbb{R}^T \).
Such spaces are characterized by the continuity of pointwise evaluation.

\begin{definition} \label{def:rkbs}
    A Banach space \( X \subseteq \mathbb{R}^T \) is called a \emph{reproducing kernel Banach space} if for every \( t \in T \), the functional \( \delta_t:X\to \mathbb{R},\; x\mapsto x(t) \) is continuous.
\end{definition}

If \( X \) is a Hilbert space it is called a \emph{reproducing kernel Hilbert space}.
Throughout this section, \( X \subseteq \mathbb{R}^T \) denotes a reproducing kernel Banach space with norm \( \|\cdot\| \), and \( \delta_t : X \to \mathbb{R} \) denotes the evaluation functional at \( t \in T \).

We show that covariance operators of weak second-order probability measures are uniquely determined by symmetric positive-definite functions and characterize those functions that arise as kernels of Gaussian covariance operators.
For a discussion of the role of reproducing kernel Banach spaces in machine learning, we refer to \cite{bartolucci2023} and the references therein.

\begin{definition} \label{def:reproducing:kernel}
    Let \( R:X^*\to X \) be a symmetric positive operator.
    The \emph{kernel} of \( R \) is the function
    \begin{equation} \label{eq:reproducing:kernel}
        K : T \times T \to \mathbb{R}, \quad (s,t) \mapsto \langle R\delta_s, \delta_t \rangle.
    \end{equation}
\end{definition}

It is clear that the kernel of \( R \) is a well-defined symmetric positive-definite function and that \( R \delta_t = K(\cdot,t) \) for all \( t \in T \).
Moreover, since \( \Gamma = \{ \delta_t : t \in T \} \) separates points and \( R \) is symmetric, \( R \) is uniquely determined by its kernel \( K \).
In analogy with random processes, when \( R \) is a covariance operator of a weak second-order probability measure \( \mu \) on \( (X,\mathcal{C}(X)) \), we call \( K \) the \emph{covariance function} of \( \mu \).

Not every symmetric positive-definite function determines a covariance operator of a probability measure, nor does it necessarily determine a symmetric positive operator.
To obtain necessary and sufficient conditions for the existence of such an operator, we recall the celebrated Moore--Aronszajn theorem \cite{aronszajn1950}.

\begin{theorem}[Moore--Aronszajn] \label{thm:Moore:Aronszajn}
    Let \( K : T \times T \to \mathbb{R} \) be a symmetric positive-definite function.
    Then there exists a unique Hilbert space \( H \subseteq \mathbb{R}^T \) such that
    \begin{enumerate}
        \item \label{thm:Moore:Aronszajn:inclusion}
        \( K(\cdot,t) \in H \) for all \( t \in T \);

        \item \label{thm:Moore:Aronszajn:reproducing}
        \( (x,K(\cdot,t))_H = x(t) \) for all \( x \in H \) and \( t \in T \).
    \end{enumerate}
\end{theorem}

It follows from property (\ref{thm:Moore:Aronszajn:reproducing}) that \( H \) is a reproducing kernel Hilbert space.
For a symmetric positive-definite function \( K:T \times T \to \mathbb{R} \), we denote the corresponding Hilbert space by \( H(K) \), and \( K \) is called the \emph{reproducing kernel} of \( H(K) \).
The notation and terminology are justified by the following results.

\begin{lemma}[Kernel] \label{lem:kernel}
    Let \( R : X^* \to X \) be a symmetric positive operator with kernel \( K : T \times T \to \mathbb{R} \).
    Then \( H(R) = H(K) \).
\end{lemma}
\begin{proof}
    Let \( H(R) \) be the Hilbert space associated with \( R \) in Proposition~\ref{prop:cameron:martin}.
    From \( R\delta_t = K(\cdot,t) \) and \( R(X^*) \subseteq H(R) \), we obtain that \( K(\cdot,t) \in H(R) \) for all \( t \in T \).
    Moreover, by the reproducing property \eqref{eq:reproducing:property}, it follows that
    \begin{equation*}
        (x, R\delta_t)_{H(R)} = (x, K(\cdot,t))_{H(R)} = \langle \iota x, \delta_t \rangle = \delta_t(x) = x(t), \qquad \forall x \in H(R), \, t \in T.
    \end{equation*}
    Hence, \( H(R) \) satisfies the properties of Theorem~\ref{thm:Moore:Aronszajn}, and thus \( H(R) = H(K) \) by uniqueness of the spaces \( H(R) \) and \( H(K) \).
\end{proof}

The previous lemma suggests that a symmetric positive operator factors through the Hilbert space associated with its kernel in the sense of Lemma~\ref{lem:factorization}.
This leads to the following characterization of kernels of symmetric positive operators.

\begin{proposition} \label{prop:kernel}
    Let \( K : T \times T \to \mathbb{R} \) be a symmetric positive-definite function.
    The following are equivalent:
    \begin{enumerate}
        \item There exists a symmetric positive operator \( R : X^* \to X \) with kernel \( K \).
        \item \( H(K) \subseteq X \) and the embedding \( \iota : H(K) \to X \) is continuous.
    \end{enumerate}
\end{proposition}
\begin{proof}
    Assume first that there exists a symmetric positive operator \( R : X^* \to X \) with kernel \( K \).
    By Lemma~\ref{lem:kernel}, we obtain \( H(R) = H(K) \).
    Hence \( H(K) \subseteq X \), and the embedding \( \iota : H(K) \to X \) is continuous by Proposition~\ref{prop:cameron:martin}.

    Conversely, suppose that \( H(K) \subseteq X \) and that the embedding \( \iota : H(K) \to X \) is continuous.
    Then the adjoint operator \( \iota^{*} : X^{*} \to H(K)^* \) exists and satisfies
    \begin{equation*}
        \langle x, \iota^* \delta_t \rangle_{H(K)} = \langle \iota x, \delta_t \rangle_X = x(t) = (x, K(\cdot,t))_{H(K)}, \qquad \forall x \in H(K),\, t \in T.
    \end{equation*}
    By the Riesz representation theorem, \( \iota^* \delta_t \) can be identified with \( K(\cdot,t) \) for all \( t \in T \), and hence \( R = \iota \iota^* \).
    It is straightforward to verify that \( R \) is a symmetric positive operator with kernel \( K \).
\end{proof}

\begin{remark}
    By the closed graph theorem, the condition \( H(K) \subseteq X \) implies that the embedding \( \iota : H(K) \to X \) is continuous, see \cite[p.~382]{aronszajn1950}.
\end{remark}

The previous proposition is the Banach space analogue of \cite[Theorem~1.1]{lukic2001}.
The condition \( H(K) \subseteq X \) is the \emph{dominance} assumption, and the symmetric positive operator \( R \) induced by \( \iota \) is the corresponding \emph{dominance operator}.

Although covariance operators of Radon Gaussian measures on \( X \) are nuclear operators, the converse is not true in the Banach space setting.
Moreover, nuclear dominance concerns the operator \( R \) itself rather than the embedding \( \iota : H(K) \to X \).
From the perspective of constructing (Radon) Gaussian measures on \( X \) from \( K \), it is arguably more natural to impose conditions on the embedding \( \iota \) rather than on the induced operator \( R \).

\begin{theorem} \label{thm:covariance:function}
    Let \( K : T \times T \to \mathbb{R} \) be a symmetric positive-definite function.
    Then there exists a centered Radon Gaussian measure on \( (X,\mathcal{B}(X)) \) with covariance function \( K \) if and only if \( H(K) \subseteq X \) and \( \iota : H(K) \to X \) is \( \gamma \)-radonifying.
\end{theorem}
\begin{proof}
    Suppose first that there exists a Radon Gaussian measure \( \gamma \) on \( X \) with covariance function \( K \), and let \( R : X^* \to X \) be its covariance operator.
    Then \( H(R) \) coincides with \( H(K) \), and by Theorem~\ref{thm:gaussian:covariance} the embedding \( \iota \) is \( \gamma \)-radonifying.

    Conversely, suppose that \( H(K) \subseteq X \) and that \( \iota \) is \( \gamma \)-radonifying.
    Let \( R = \iota \iota^* \).
    Then \( R \) is a symmetric positive operator with \( H(R) = H(K) \).
    By Theorem~\ref{thm:gaussian:covariance}, \( R \) is the covariance operator of a centered Radon Gaussian measure \( \gamma \) on \( (X,\mathcal{B}(X)) \).
    The covariance function of \( \gamma \) is \( K \).
\end{proof}

In the remainder of this section we establish several properties of reproducing kernel Banach spaces that will be used to address measurability issues for Gaussian processes by restricting to countable index sets.
For a subset \( S \subseteq T \) and \( x \in \mathbb{R}^T \), we write \( x|_S \) for the restriction of \( x \) to \( S \).

\begin{theorem} \label{thm:restriction}
    Let \( S \subseteq T \) be a non-empty subset.
    Then \( X(S) = \{ x|_S : x \in X \} \) is a reproducing kernel Banach space with respect to the norm
    \begin{equation} \label{eq:restriction}
        \| y \|_{X(S)} =
        \inf_{\substack{
            x \in X \\
            x|_S = y
        }} \| x \|_X.
    \end{equation}
\end{theorem}
\begin{proof}
    Let \( N = \bigcap_{s \in S} \{ x \in X : \delta_s(x) = 0 \} \), where \( \delta_s : X \to \mathbb{R}, \, x \mapsto x(s) \) is the evaluation functional at \( s \).
    By continuity of the evaluation functionals, \( N \) is the intersection of closed subspaces of \( X \), and thus is itself a closed subspace of \( X \).

    Let \( X / N \) be the quotient space, and let \( \pi: X \to X/N, \, x \mapsto x + N \) be the quotient map.
    It follows from \cite[Theorem 1.41]{rudin1991} that \( X/N \) is a Banach space with respect to the norm
    \begin{equation*}
        \| \pi(x) \|_{X/N} = \inf_{z \in N}\| x-z \|_X.
    \end{equation*}

    Define \( \Phi: X/N \to X(S),\, x + N \mapsto x|_S \).
    It is clear that \( \Phi \) is a well-defined linear operator, because if \( x+N = z+N \), then \( x|_S = z|_S \) by definition of \( N \).
    Moreover, \( \Phi \circ \pi(x)=x|_S \) for all \( x\in X \), so \( \Phi \) is surjective; and if \( \Phi(x+N)=0 \), then \( x|_S=0 \), whence \( x \in N \) and \( x+N = 0 \), so \( \Phi \) is injective.

    Therefore \( \Phi \) is a linear isomorphism, and \( X(S) \) inherits a Banach space structure from \( X/N \) with the norm given by \( \|y\|_{X(S)} = \|\Phi^{-1}(y)\|_{X/N} \).
    In particular, for \( y = \Phi\circ\pi(x) \), it follows that \( x|_S = y \), and \( z|_S = y \) if and only if \( z \in x+N \).
    Thus,
    \begin{equation*}
        \| y \|_{X(S)} = \| \pi(x) \|_{X / N} = \inf_{z \in N} \| x-z \|_X = \inf_{\substack{
            x \in X \\
            x|_S = y
        }} \| x \|_X.
    \end{equation*}

    Finally, let \( \tilde{\delta}_s: X(S) \to \mathbb{R}, \, y \mapsto y(s) \) be the evaluation functional at \( s \in S \).
    Then, by continuity of \( \delta_s \), there exists a constant \( C_s > 0 \) such that
    \begin{equation*}
        |\tilde{\delta}_s(y)| = |y(s)| = |x(s)| = |\delta_s(x)| \leq C_s \| x \|_X, \qquad \forall x\in X,\; x|_S = y.
    \end{equation*}
    Taking the infimum over all \( x \in X \) with \( x|_S = y \) shows that \( |\tilde{\delta}_s(y)| \leq C_s \| y \|_{X(S)} \).
    Hence \( \tilde{\delta}_s \) is bounded and thus continuous for every \( s \in S \).
\end{proof}

In the Hilbert space case the restriction of \( H(K) \) to a non-empty subset \( S \subseteq T \) is again a reproducing kernel Hilbert space with kernel \( K|_{S \times S} \).
Hence Theorem~\ref{thm:restriction} is a generalization of the construction considered by Aronszajn \cite[p.~351]{aronszajn1950}.

For a subset \( S \subseteq T \), we denote by \( X(S) \) the restriction of \( X \) to \( S \) equipped with the norm \eqref{eq:restriction}, and by \( \pi_S : X \to X(S), \, x \mapsto x|_S \) the corresponding restriction map.
A subset \( S \) is called a \emph{determining set} for \( X \) if \( \pi_S \) is an isometric isomorphism.
This terminology goes back to Fortet~\cite{fortet1974}.
Theorem~\ref{thm:restriction} yields the following equivalent conditions for \( S \) to be a determining set for \( X \).

\begin{corollary} \label{cor:determining:set}
    Let \( S \subseteq T \).
    The following are equivalent:
    \begin{enumerate}
        \item \( S \) is a determining set for \( X \).
        \item The restriction \( \pi_S: X \to X(S) \) is injective.
        \item The family \( \{ \delta_s : s \in S \} \) separates points of \( X \).
    \end{enumerate}
\end{corollary}

To resolve measurability issues it is of principal interest to determine when a space \( X \) admits a countable determining set.
A sufficient condition is provided by the following lemma.
For the proof, recall that a subset \( \Delta \subseteq X^* \) is said to be weak-\(*\) sequentially dense in \( \Gamma \subseteq X^* \) if for every \( f \in \Gamma \) there exists a sequence \( (f_k) \subseteq \Delta \) such that \( f_k \to f \) in the weak-\( * \) topology.
Further details on weak-\( * \) sequential density and closures can be found in \cite{banach1987,ostrovskii2001}.

\begin{lemma} \label{lem:determining:set}
    Suppose that \( X \) is separable.
    Then there exists a countable determining set \( S \subseteq T \) for \( X \).
\end{lemma}
\begin{proof}
    Let \( \Gamma = \{ \delta_t : t \in T \} \subseteq X^* \).
    Because \( X \) is separable, there exists a countable subset \( \Delta \subseteq \Gamma \) that is weak-\( * \) sequentially dense in \( \Gamma \); see \cite[p.~76]{banach1987}.
    As each element of \( \Delta \) is of the form \( \delta_t \), there is a countable set \( S \subseteq T \) such that \( \Delta = \{ \delta_s : s \in S \} \).

    Let \( x \in X \) and suppose that \( x(s) = 0 \) for all \( s \in S \).
    For every \( t \in T \), there exists a sequence \( (s_k) \subseteq S \) such that \( \delta_{s_k} \to \delta_t \) in the weak-\( * \) topology.
    Then
    \begin{equation*}
        x(t) = \delta_t(x) = \lim_{k\to\infty}\delta_{s_k}(x) = \lim_{k\to\infty}x(s_k) = 0.
    \end{equation*}
    Hence \( x(t) = 0 \) for all \( t \in T \), so \( x = 0 \).
    Thus \( S \) is a determining set for \( X \).
\end{proof}

\begin{remark} \label{rem:determining:set}
    A simpler measure-theoretic proof can be obtained by noting that \( \mathcal{C}(X,\Gamma) = \mathcal{B}(X) \) when \( X \) is separable.
    In particular, \( \mathcal{C}(X,\Gamma) \) contains all singletons, and hence \( \Gamma \) contains a countable separating subset; see \cite[Lemma~I.2.1]{vakhania1987}.
\end{remark}

When \( X \) is a Hilbert space with reproducing kernel \( K \), the converse also holds: if \( X \) has a countable determining set, then \( X \) is separable.
In preparation for this result, we introduce the \emph{kernel metric} on \( T \), defined by
\begin{equation} \label{eq:kernel:metric}
    d_K : T \times T \to \mathbb{R}, \qquad
    (s,t) \mapsto \| K(\cdot,s) - K(\cdot,t) \|.
\end{equation}
It is readily verified that \( d_K \) defines a pseudo-metric on \( T \).
Moreover, by the reproducing property and the polarization identity, it follows that
\begin{equation} \label{eq:kernel:metric:Loeve}
    d_K(s,t)
    = \sqrt{K(s,s) + K(t,t) - 2K(s,t)},
    \qquad \forall s,t \in T.
\end{equation}

The following lemma extends \cite[Lemma~4.3]{lukic2001} by including the equivalent condition~(\ref{lem:rkhs:separability:determining}).
In \cite{lukic2001}, the proof of the implication (\ref{lem:rkhs:separability:separable})\( \implies \)(\ref{lem:rkhs:separability:TdK}) relies on \cite[Theorem~1.2]{fortet1973}, which is stated without proof.
We provide here a self-contained argument.

\begin{lemma} \label{lem:rkhs:separability}
    Let \( K : T \times T \to \mathbb{R} \) be a symmetric positive-definite function.
    Then the following are equivalent:
    \begin{enumerate}
        \item \label{lem:rkhs:separability:separable}
        \( H(K) \) is separable.

        \item \label{lem:rkhs:separability:TdK}
        \( (T,d_K) \) is separable.

        \item \label{lem:rkhs:separability:determining}
        \( H(K) \) has a countable determining set.
    \end{enumerate}
\end{lemma}
\begin{proof}
    (\ref{lem:rkhs:separability:separable}) \( \implies \) (\ref{lem:rkhs:separability:TdK})
    Because \( H(K) \) is a separable metric space, every subset of \( H(K) \) is separable.
    In particular, the set \( \{ K(\cdot,t) : t \in T \} \) is separable.
    Hence there exists a countable set \( S \subseteq T \) such that \( \{ K(\cdot,s) : s \in S \} \) is dense in \( \{ K(\cdot,t) : t \in T \} \).
    Thus for every \( t\in T \) and every \( \epsilon>0 \) there exists \( s\in S \) with
    \begin{equation*}
        d_K(s,t) = \| K(\cdot,s) - K(\cdot,t) \|_{H(K)} < \epsilon.
    \end{equation*}
    Therefore \( S \) is a countable dense subset of \( (T,d_K) \), and hence \( (T,d_K) \) is separable.

    \smallskip
    (\ref{lem:rkhs:separability:TdK}) \( \implies \) (\ref{lem:rkhs:separability:determining})
    Let \( S \subseteq T \) be a countable subset which is dense in \( (T,d_K) \).
    Suppose that \( x \in H(K) \) satisfies \( x|_S = 0 \), and let \( t \in T \) be fixed.
    For every \( \epsilon > 0 \), there exists \( s \in S \) such that \( d_K(s,t) < \epsilon \).
    By the reproducing property,
    \begin{equation*}
        |x(t)| = |(x,K(\cdot,t))| = |(x,K(\cdot,t)-K(\cdot,s))| \leq d_K(s,t)\| x \|_{H(K)} < \epsilon \| x \|_{H(K)}.
    \end{equation*}
    Letting \( \epsilon \to 0 \) yields \( x(t)=0 \).
    Because \( t \) was arbitrary, it follows that \( x(t) = 0 \) for all \( t \in T \), and hence \( S \) is a countable determining set for \( H(K) \).

    \smallskip
    (\ref{lem:rkhs:separability:determining}) \( \implies \) (\ref{lem:rkhs:separability:separable})
    Let \( S \subseteq T\) be a countable determining set for \( H(K) \).
    Suppose that \( x \in H(K) \) is orthogonal to the family \( \Gamma = \{ K(\cdot,s) : s \in S \} \).
    Then by the reproducing property,
    \begin{equation*}
        x(s) = (x, K(\cdot,s)) = 0, \qquad \forall s \in S.
    \end{equation*}
    Thus \( x|_S = 0 \), and since \( S \) is a determining set, it follows that \( x = 0 \).
    Hence \( \Gamma \) is a countable total subset of \( H(K) \) and therefore \( H(K) \) is separable.
\end{proof}

For reproducing kernel Banach spaces, a similar characterization is not available.
Indeed, the following example shows that a reproducing kernel Banach space with a countable determining set need not be separable.

\begin{example} \label{ex:determining:set}
    The space \( C_b(\mathbb{R}) \) of bounded continuous functions on \( \mathbb{R} \) is a non-separable reproducing kernel Banach space with countable determining set \( \mathbb{Q} \).
\end{example}

Fortet's condition \cite[Theorem~1.1]{fortet1973} characterizes the functions in a reproducing kernel Hilbert space by taking restrictions to finite subsets of \( T \).
A similar result holds for reproducing kernel Banach spaces admitting a countable determining set.

\begin{lemma} \label{lem:rkbs:countable:restriction}
    Let \( Y \subseteq \mathbb{R}^T \) be a reproducing kernel Banach space and suppose that \( X \) admits a countable determining set.
    Then the following are equivalent:
    \begin{enumerate}
        \item \label{lem:rkbs:countable:restriction:inclusion}
        \( Y \) is a vector subspace of \( X \).

        \item \label{lem:rkbs:countable:restriction:restricted}
        \( Y(S) \subseteq X(S) \) for every countable subset \( S \subseteq T \).

        \item \label{lem:rkbs:countable:restriction:restricted:determining}
        \( Y(S) \subseteq X(S) \) for every countable determining set \( S \subseteq T \) of \( X \).
    \end{enumerate}
\end{lemma}
\begin{proof}
    The implications \( (\ref{lem:rkbs:countable:restriction:inclusion}) \implies (\ref{lem:rkbs:countable:restriction:restricted}) \) and \( (\ref{lem:rkbs:countable:restriction:restricted}) \implies (\ref{lem:rkbs:countable:restriction:restricted:determining}) \) are immediate.
    For the implication \( (\ref{lem:rkbs:countable:restriction:restricted:determining}) \implies (\ref{lem:rkbs:countable:restriction:inclusion}) \), let \( S \subseteq T \) be a countable determining set for \( X \) and for \( t \in T \) let \( S_t = S \cup \{ t \} \).
    Then \( S_t \) is again a countable determining set for \( X \), so that \( Y(S_t) \subseteq X(S_t) \) by hypothesis.

    We verify that \( Y \subseteq X \).
    For \( y \in Y \) we have \( y|_S \in Y(S) \subseteq X(S) \), so there exists \( x \in X \) with \( x|_S = y|_S \), and \( x \) is unique because \( S \) is a determining set for \( X \).
    Let \( x_t \in X \) be such that \( x_t|_{S_t} = y|_{S_t} \).
    Then \( x_t|_S = y|_S = x|_S \), and by uniqueness of \( x \) it follows that \( x_t = x \) for every \( t \in T \).
    Hence \( x(t) = x_t(t) = y(t) \) for every \( t \in T \), and thus \( y = x \in X \).
\end{proof}

\section{Driscoll's Theorem}
\label{sec:Driscoll:theorem}

We are now in a position to establish the connection between Gaussian random elements in a reproducing kernel Banach space \( X \) and Gaussian random processes.
In particular, we extend the results of Lukić and Beder~\cite[Section~7]{lukic2001} to the Banach space setting and provide necessary and sufficient conditions under which a Gaussian process defines a Gaussian random element in \( X \).
While it is immediate that every Gaussian random element in \( X \) induces a Gaussian process, the converse is considerably more delicate and involves subtle measurability issues.
The main results are Theorems~\ref{thm:Driscoll:one} and~\ref{thm:Driscoll:zero}, which together yield a reproducing kernel Banach space generalization of the classical Driscoll theorem.

Throughout this section, \( X \subseteq \mathbb{R}^T \) denotes a separable reproducing kernel Banach space and \( (\Omega,\mathcal{A},\mu) \) a probability space.
For \( t \in T \), we write \( \delta_t : \mathbb{R}^T \to \mathbb{R},\, x \mapsto x(t) \), for the evaluation functional at \( t \).
Because \( X \) is continuously embedded into \( \mathbb{R}^T \), we use the same notation for the restriction of \( \delta_t \) to \( X \).
Similarly, for \( S \subseteq T \) we write \( \pi_S : \mathbb{R}^T \to \mathbb{R}^S, \, x \mapsto x|_S \) for the restriction map, and use the same notation for its restriction to \( X \).
A random element \( \eta : \Omega \to X \) will, without further comment, also be regarded as a random process over \( T \).
When \( \eta \) is viewed as a random element in \( \mathbb{R}^T \) we write \( \eta_t = \langle \eta, \delta_t \rangle \) for its coordinate at \( t \).

We begin by addressing the subtle measurability issues arising when passing from a random process to an \( X \)-valued random element.

\begin{lemma} \label{lem:rkbs:nonmeasurability}
    If \( T \) is uncountable and \( C \in \mathcal{C}(\mathbb{R}^T) \) satisfies \( C \subseteq X \), then \( C = \emptyset \).
\end{lemma}
\begin{proof}
    Suppose, for a contradiction, that there exists a nonempty \( C \in \mathcal{C}(\mathbb{R}^T) \) such that \( C \subseteq X \).
    Let \( \Gamma = \{ \delta_t : t \in T \} \subseteq (\mathbb{R}^T)^* \).
    Since \( \operatorname{span}\Gamma = (\mathbb{R}^T)^* \) and finite sums of measurable functions are measurable, it follows that \( \mathcal{C}(\mathbb{R}^T,\Gamma) = \mathcal{C}(\mathbb{R}^T) \).
    Therefore by \cite[Proposition~I.1.8]{vakhania1987}, there exist a countable subset \( S \subseteq T \) and a Borel set \( B \in \mathcal{B}(\mathbb{R}^S) \) such that \( C = \pi_S^{-1}(B) \).
    We may assume that \( S \) contains a countable determining set for \( X \).
    Indeed, \( X \) admits a countable determining set \( S_0 \) by Lemma~\ref{lem:determining:set}, and replacing \( S \) by \( S \cup S_0 \) and \( B \) by \( B \times \mathbb{R}^{S_0 \setminus S} \) leaves \( C \) unchanged.

    The kernel \( N = \ker(\pi_S) = \{ x \in \mathbb{R}^T : x|_S = 0 \} \) is an infinite-dimensional subspace of \( \mathbb{R}^T \) since \( T \setminus S \) is uncountable.
    Moreover, because \( C \subseteq X \) and \( C + N = C \) it follows that \( N \subseteq X \).
    Since \( \pi_S \) is injective on \( X \) by Corollary~\ref{cor:determining:set}, we obtain that \( N = N \cap X = \{ 0 \} \), contradicting that \( N \) is infinite-dimensional.
\end{proof}

In particular, the lemma shows that \( X \notin \mathcal{C}(\mathbb{R}^T) \) whenever \( T \) is uncountable and that \( X \) belongs to the Lebesgue completion of \( \mathcal{C}(\mathbb{R}^T) \) with respect to any probability measure precisely when \( X \) has outer measure zero.
As a consequence, there is no guarantee that
\begin{equation*}
    \{ \omega \in \Omega : \xi(\cdot,\omega) \in X \} \in \mathcal{A}.
\end{equation*}
To address the measurability issues arising from this fact, we avoid statements of the form ``\( \xi \) has sample paths almost surely in \( X \)'' and instead work with a measurable set \( \Omega_0 \in \mathcal{A} \) with \( \mu(\Omega_0) = 1 \) such that \( \xi(\cdot,\omega) \in X \) for all \( \omega \in \Omega_0 \).

\begin{lemma} \label{lem:random:process:restriction}
    Let \( \xi : T \times \Omega \to \mathbb{R} \) be a random process, and suppose that there exists \( \Omega_0 \in \mathcal{A} \) with \( \mu(\Omega_0) = 1 \) such that \( \xi(\cdot,\omega) \in X \) for all \( \omega \in \Omega_0 \).
    Then there exists a Borel random element \( \eta : \Omega \to X \) such that \( \xi_t = \eta_t \) almost surely for all \( t \in T \).
\end{lemma}
\begin{proof}
    Define
    \begin{equation*}
        \eta : \Omega \to X, \qquad
        \omega \mapsto
        \begin{cases}
            \xi(\cdot,\omega), & \omega \in \Omega_0, \\
            0, & \omega \notin \Omega_0.
        \end{cases}
    \end{equation*}
    This is well-defined because \( \xi(\cdot,\omega) \in X \) for all \( \omega \in \Omega_0 \) and \( 0 \in X \).
    Moreover, for every \( t \in T \) and every Borel set \( B \in \mathcal{B}(\mathbb{R}) \),
    \begin{equation*}
        \eta_t^{-1}(B) =
        \begin{cases}
            \xi_t^{-1}(B) \cap \Omega_0, & 0 \notin B, \\
            (\xi_t^{-1}(B) \cap \Omega_0) \cup (\Omega \setminus \Omega_0), & 0 \in B.
        \end{cases}
    \end{equation*}
    Hence the functions \( \eta_t = \langle \eta, \delta_t \rangle \), \( t \in T \), are random variables.
    Since \( X \) is separable and \( \Gamma = \{ \delta_t : t \in T \} \) separates points in \( X \), it follows from \cite[Theorem II.1.1]{vakhania1987} that \( \eta \) is a Borel random element in \( X \).
    Finally, \( \Omega_0 \subseteq \{ \omega \in \Omega : \eta_t(\omega) = \xi_t(\omega) \} \), and hence \( \mu(\eta_t = \xi_t)=1 \) for all \( t \in T \).
\end{proof}

The random element \( \eta \) is called the \emph{random element in \( X \) defined by \( \xi \)}.
It is clear that \( \eta \) is uniquely defined up to a set of probability zero.
A natural question is whether \( \eta \) is Gaussian whenever \( \xi \) is a centered Gaussian random process.

In \cite[Theorem~7.1]{lukic2001}, this is proved when \( X \) is a Hilbert space, using the weak-\( * \) sequential density of \( \Gamma = \operatorname{span}\{ \delta_t : t \in T \} \subseteq X^* \) in \( X^* \).
However, when \( X \) is a Banach space, \( \Gamma \) need not be weak-\( * \) sequentially dense in \( X^* \) unless \( X \) is reflexive or, more generally, quasi-reflexive; see \cite[Theorem~2.3]{ostrovskii2001}.
Therefore, to obtain a Banach space analogue, we instead appeal to the characterization of centered Gaussian measures through their rotational invariance.

\begin{theorem} \label{thm:GP:restriction}
    Let \( \xi : T \times \Omega \to \mathbb{R} \) be a centered Gaussian process, and suppose that there exists \( \Omega_0 \in \mathcal{A} \) with \( \mu(\Omega_0) = 1 \) such that \( \xi(\cdot,\omega) \in X \) for all \( \omega \in \Omega_0 \).
    Then the random element \( \eta : \Omega \to X \) defined by \( \xi \) is a centered Gaussian random element.
\end{theorem}
\begin{proof}
    Let \( \nu \) denote the distribution of \( \eta \), and for \( \varphi \in \mathbb{R} \) define
    \begin{equation*}
        Q_\varphi : X \times X \to X, \qquad (x,y) \mapsto \sin{\varphi} \, x + \cos{\varphi} \, y.
    \end{equation*}
    Each \( Q_\varphi \) is a continuous linear operator and thus measurable.
    Set \( \nu_\varphi = (\nu \otimes \nu)\circ Q_\varphi^{-1} \).
    Then \( \nu_0=\nu \) and for \( f \in \operatorname{span}\{ \delta_t : t \in T \} \) with \( f = \sum_{t \in T_0} c_t \delta_t \), we obtain
    \begin{align*}
        \widehat{\nu}_\varphi(f)
        &= \int_{X \times X} \exp(i(f(\sin{\varphi} x + \cos{\varphi} y))) \nu\otimes\nu(d(x,y)) \\
        &= \int_{\Omega} \exp(i\sin{\varphi} f(\eta(\omega)))\mu(d\omega) \int_{\Omega} \exp(i\cos{\varphi} f(\eta(\omega)))\mu(d\omega) \\
        &= \int_{\Omega} \exp\left(i\sin{\varphi} \sum_{t\in T_0} c_t \xi_t(\omega)\right)\mu(d\omega) \int_{\Omega} \exp\left(i\cos{\varphi} \sum_{t\in T_0} c_t \xi_t(\omega)\right)\mu(d\omega)
    \end{align*}
    where the last equality uses \( \eta_t = \xi_t \) almost surely for all \( t \in T \).
    Since \( \xi \) is a centered Gaussian process with covariance function \( K \), it follows that
    \begin{equation*}
        \widehat{\nu}_\varphi(f)
        = \exp{\left(-\frac{1}{2}\left(\sin^2{\varphi}+\cos^2{\varphi}\right)\sum_{s \in T_0}\sum_{t \in T_0} c_sc_tK(s,t)\right)}
        = \widehat{\nu}(f),
        \qquad \forall \varphi\in\mathbb{R}.
    \end{equation*}
    Because \( f \) was arbitrary, the identity holds for all \( f \in \operatorname{span}\{ \delta_t : t \in T \} \), which is a separating subspace of \( X^* \).
    Therefore \cite[Theorem~IV.2.2]{vakhania1987} implies that \( \nu_\varphi=\nu \) on \( (X,\mathcal{B}(X)) \) for all \( \varphi\in\mathbb{R} \).
    Hence \( \nu \) is rotation invariant and therefore a centered Gaussian measure; see \cite[Proposition~2.2.10]{bogachev1998}.
\end{proof}

\begin{remark} \label{rem:Driscoll:restriction}
    Separability of \( X \) can be replaced by (quasi-)reflexivity, using the weak-\( * \) sequential density of \( \operatorname{span}\{ \delta_t : t \in T \} \) in \( X^* \); see \cite[Theorem~2.3]{ostrovskii2001}.
\end{remark}

Depending on the probability space \( (\Omega,\mathcal{A},\mu) \), the assumption that there exists \( \Omega_0 \) with \( \mu(\Omega_0) = 1 \) and \( \xi(\cdot,\omega) \in X \) for all \( \omega \in \Omega_0 \) can be rather restrictive, as shown by Lemma~\ref{lem:rkbs:nonmeasurability}.
Although a random process may fail to satisfy this assumption, it may nevertheless admit a modification that does.
This was proved in the Hilbert space setting in \cite[Theorem~7.2]{lukic2001}.
In preparation for extending this result to Banach spaces, we establish the following lemma.

\begin{lemma} \label{lem:rkbs:measurability}
    If \( T \) is countable, then \( \mathcal{B}(X) \subseteq \mathcal{C}(\mathbb{R}^T) \).
\end{lemma}
\begin{proof}
    Because \( T \) is countable, it follows that \( \mathbb{R}^T \) is a separable Fréchet space, and thus \( \mathcal{B}(\mathbb{R}^T) = \mathcal{C}(\mathbb{R}^T) \).
    Moreover, the embedding \( \iota : X \to \mathbb{R}^T \) is continuous by continuity of point evaluations in \( X \).
    Hence \( \iota \) is an injective Borel map between complete separable metric spaces, and by the Kuratowski theorem \cite[Theorem~I.1.1]{vakhania1987} it follows that \( \iota(B) = B \in \mathcal{B}(\mathbb{R}^T) \) for all \( B \in \mathcal{B}(X) \).
\end{proof}

For the proof of the Banach space analogue of \cite[Theorem~7.2]{lukic2001}, we follow the original approach of \cite{driscoll1973}, replacing the Prokhorov theorem \cite[Theorem~IV.2.4]{vakhania1987}, which characterizes covariance operators of Gaussian measures on Hilbert spaces, with the corresponding Banach space version in Theorem~\ref{thm:gaussian:covariance}.

\begin{theorem} \label{thm:Driscoll:one}
    Let \( \xi : T \times \Omega \to \mathbb{R} \) be a centered Gaussian process with covariance function \( K \).
    Then the following are equivalent:
    \begin{enumerate}
        \item \label{thm:Driscoll:one:modification}
        There exists a centered Gaussian random element \( \eta : \Omega \to X \) such that \( \xi_t = \eta_t \) almost surely for all \( t \in T \).

        \item \label{thm:Driscoll:one:radonifying}
        \( H(K) \subseteq X \) and the embedding \( \iota : H(K) \to X \) is \( \gamma \)-radonifying.
    \end{enumerate}
\end{theorem}
\begin{proof}
    (\ref{thm:Driscoll:one:modification}) \( \implies \) (\ref{thm:Driscoll:one:radonifying}).
    Let \( \gamma \) denote the distribution of \( \eta \).
    Since \( X \) is separable, \( \gamma \) is a Radon Gaussian measure.
    Hence \( H(K) \subseteq X \), and the embedding \( \iota : H(K) \to X \) is \( \gamma \)-radonifying by Theorem~\ref{thm:covariance:function}.

    \smallskip
    (\ref{thm:Driscoll:one:radonifying}) \( \implies \) (\ref{thm:Driscoll:one:modification}).
    Let \( S_0 \subseteq T \) be a countable determining set for \( X \).
    Since \( H(K) \subseteq X \), the same set \( S_0 \) is determining for \( H(K) \).
    Hence \( H(K) \) is separable, and there exists a countable set \( S_1 \subseteq T \) dense in \( (T,d_K) \).
    Set \( S = S_0 \cup S_1 \); then \( S \) is countable and \( \pi_S \) is continuous and linear.

    Let \( \mu^\xi \) denote the distribution of \( \xi \), and let \( \gamma \) be the centered Gaussian measure on \( (X,\mathcal{B}(X)) \) with covariance function \( K \); such a measure exists by Theorem~\ref{thm:covariance:function}.
    Since \( X \) embeds continuously into \( \mathbb{R}^T \) and centered Gaussian measures are uniquely determined by their covariances, it follows that
    \begin{equation*}
        \gamma(C \cap X) = \mu^\xi(C),
        \qquad \forall C \in \mathcal{C}(\mathbb{R}^T).
    \end{equation*}
    In particular, this holds for \( C = \pi_S^{-1}(X(S)) \), which belongs to \( \mathcal{C}(\mathbb{R}^T) \) because \( \pi_S \) is measurable and \( X(S) \in \mathcal{C}(\mathbb{R}^S) \) by Lemma~\ref{lem:rkbs:measurability}.
    Thus
    \begin{equation*}
        \gamma(X) = \gamma(\pi_S^{-1}(X(S)) \cap X) = \mu^\xi(\pi_S^{-1}(X(S))) = \mu(\{ \omega \in \Omega : \widetilde{\xi}(\cdot,\omega) \in X(S) \}) = 1,
    \end{equation*}
    where \( \widetilde{\xi} : S \times \Omega \to \mathbb{R}, \, (s,\omega) \mapsto \xi(s,\omega) \).
    Hence, by Theorem~\ref{thm:GP:restriction}, there exists a Gaussian random element \( \widetilde{\eta} : \Omega \to X(S) \) such that \( \widetilde{\eta}_s = \widetilde{\xi}_s \) almost surely for all \( s\in S \).
    Because \( S \) is a determining set for \( X \), the spaces \( X \) and \( X(S) \) are isometrically isomorphic.
    Hence there exists a unique Gaussian random element \( \eta : \Omega \to X \) satisfying \( \eta_s(\omega) = \widetilde{\eta}_s(\omega) \) for all \( \omega \in \Omega \) and \( s \in S \).
    The covariance function of \( \eta \) coincides with \( K \).

    It remains to verify that \( \xi_t = \eta_t \) almost surely for all \( t \in T \).
    Clearly this holds for all \( t \in S \).
    For arbitrary \( t \in T \) and \( s\in S \), the triangle inequality gives
    \begin{equation*}
        \| \xi_t - \eta_t \|_{L^2} \leq \| \xi_t - \xi_s \|_{L^2} + \| \xi_s - \eta_s \|_{L^2} + \| \eta_s - \eta_t \|_{L^2}.
    \end{equation*}
    The middle term vanishes because \( \eta_s = \xi_s \) almost surely.
    Moreover, the polarization identity and the fact that both \( \xi \) and \( \eta \) have covariance function \( K \) yield
    \begin{equation*}
        d_K(s,t) = \| \xi_t - \xi_s \|_{L^2} = \| \eta_t - \eta_s \|_{L^2}, \qquad \forall s \in S, \, \forall t \in T.
    \end{equation*}
    Thus
    \begin{equation*}
        \| \xi_t - \eta_t \|_{L^2} \leq 2 d_K(s,t).
    \end{equation*}
    Since \( S \) is dense in \( (T,d_K) \), the right-hand side can be made arbitrarily small.
    Hence \( \| \xi_t - \eta_t \|_{L^2} = 0 \), and thus \( \xi_t = \eta_t \) almost surely for all \( t \in T \).
\end{proof}

\begin{remark} \label{rem:Driscoll:one}
    Separability of \( X \) can be replaced by reflexivity, using the reflexive version of Theorem~\ref{thm:GP:restriction} and the fact that every probability measure \( \mu \) on \( (X,\mathcal{C}(X)) \) admits a unique Radon extension to \( (X,\mathcal{B}(X)) \); see \cite[Theorem~I.5.3]{vakhania1987}.
\end{remark}

Theorem~\ref{thm:Driscoll:one} may be viewed as the Banach space analogue of the ``one'' part of the Driscoll theorem.
A similar statement appears in \cite[Theorem~4.3]{hajek1962} in the setting of Hilbert spaces.
It is formulated in a language predating reproducing kernel Hilbert spaces and nuclear dominance.
The corresponding ``zero'' part is proved similarly to \cite[Theorem~7.3]{lukic2001}, using the Kallianpur zero--one law \cite{kallianpur1970b}.

\begin{theorem} \label{thm:Driscoll:zero}
    Let \( \xi : T \times \Omega \to \mathbb{R} \) be a centered Gaussian process with covariance function \( K \), and suppose that \( (\Omega,\mathcal{A},\mu) \) is complete.
    If \( H(K) \not\subseteq X \) or if the embedding \( \iota : H(K) \to X \) is not \( \gamma \)-radonifying, then
    \begin{equation} \label{eq:Driscoll:zero}
        \mu(\{ \omega \in \Omega : \xi(\cdot,\omega) \in X \}) = 0.
    \end{equation}
\end{theorem}
\begin{proof}
    For a countable determining set \( S \subseteq T \) of \( X \) it follows from Lemma~\ref{lem:rkbs:measurability} that \( X(S) \in \mathcal{C}(\mathbb{R}^S) \) and thus
    \begin{equation*}
        \Omega_S = \{ \omega \in \Omega : \xi(\cdot,\omega) \in \pi_S^{-1}(X(S)) \} \in \mathcal{A}.
    \end{equation*}
 
    We claim that \( \mu(\Omega_S) < 1 \) for some countable determining set \( S \subseteq T \) of \( X \).
    Indeed, suppose for a contradiction that \( \mu(\Omega_S) = 1 \) for every such \( S \).
    Then \( \xi \) defines a collection of centered Gaussian random elements \( \eta^S : \Omega \to X(S) \) with covariance function \( K|_{S \times S} \), such that \( \langle \eta^S, \delta_s \rangle = \xi_s \) almost surely for all \( s \in S \), by Theorem~\ref{thm:GP:restriction}.
    Thus Theorem~\ref{thm:covariance:function} implies that the restriction \( H(S) \) of \( H = H(K) \) satisfies \( H(S) \subseteq X(S) \) and that the embedding \( \iota_S : H(S) \to X(S) \) is \( \gamma \)-radonifying.
    Because this holds for every countable determining set \( S \subseteq T \) of \( X \), it follows by Lemma~\ref{lem:rkbs:countable:restriction} that \( H \subseteq X \).

    Consider next a fixed countable determining set \( S \subseteq T \) for \( X \).
    Since \( H \subseteq X \), the set \( S \) is also a determining set for \( H \), so that \( H(S) \) and \( X(S) \) are isometrically isomorphic to \( H \) and \( X \), respectively.
    In particular, the following diagram commutes.
    \begin{equation*}
        \xymatrix{
            H \ar[r]^{\iota} \ar[d]_{\pi_S|_H}^{\cong} & X \ar[d]^{\pi_S|_X}_{\cong} \\
            H(S) \ar[r]_{\iota_S} & X(S)
        }
    \end{equation*}
    Because \( \iota_S \) is \( \gamma \)-radonifying, the ideal property \cite[Theorem~9.1.10]{hytonen2017} of \( \gamma \)-radonifying operators implies that \( \iota \) is \( \gamma \)-radonifying, contradicting the hypothesis.

    Let \( S \subseteq T \) be a countable determining set for \( X \) such that \( \mu(\Omega_S) < 1 \).
    Because \( X(S) \in \mathcal{C}(\mathbb{R}^S) \), the Kallianpur zero--one law \cite{kallianpur1970b} implies that either \( \mu(\Omega_S) = 0 \) or \( \mu(\Omega_S) = 1 \), and hence \( \mu(\Omega_S) = 0 \).
    Since \( x \in X \) implies \( x|_S \in X(S) \), the set
    \begin{equation*}
        \Omega_0 = \{ \omega \in \Omega : \xi(\cdot,\omega) \in X \}
    \end{equation*}
    is contained in the null set \( \Omega_S \), and therefore \( \Omega_0 \in \mathcal{A} \) with \( \mu(\Omega_0) = 0 \) where we have used that \( (\Omega,\mathcal{A},\mu) \) is a complete probability space.
\end{proof}

The identity operator is \( \gamma \)-radonifying, and hence compact, precisely in finite-dimensional Hilbert spaces.
As an immediate consequence, we obtain the following classical result, originally stated by Parzen \cite[eq.~34]{parzen1963} and subsequently proved by Kallianpur~\cite[Theorem~5.1]{kallianpur1970a} and Le Page~\cite{lepage1973}.

\begin{corollary}[Parzen] \label{cor:Parzen}
    Let \( \xi : T \times \Omega \to \mathbb{R} \) be a centered Gaussian process with covariance function \( K \), and suppose that \( (\Omega,\mathcal{A},\mu) \) is complete.
    If \( H(K) \) is infinite-dimensional and separable then
    \begin{equation} \label{eq:Parzen}
        \mu(\{ \omega \in \Omega : \xi(\cdot,\omega) \in H(K) \}) = 0.
    \end{equation}
\end{corollary}

\providecommand{\bysame}{\leavevmode\hbox to3em{\hrulefill}\thinspace}
\providecommand{\MR}{\relax\ifhmode\unskip\space\fi MR }
\providecommand{\MRhref}[2]{%
  \href{http://www.ams.org/mathscinet-getitem?mr=#1}{#2}
}
\providecommand{\href}[2]{#2}


\begin{thebibliography}{10}

\bibitem{aronszajn1950}
N. Aronszajn, \emph{Theory of reproducing kernels}, Trans. Amer. Math. Soc. \textbf{68} (1950), 337--404. \MR{51437}

\bibitem{banach1987}
S. Banach, \emph{Theory of linear operations}, North-Holland Mathematical Library, vol.~38, North-Holland Publishing Co., Amsterdam, 1987, Translated from the French by F. Jellett, With comments by A. Pe\l czy\'nski and Cz.\ Bessaga. \MR{880204}

\bibitem{bartolucci2023}
F. Bartolucci, E. De~Vito, L. Rosasco, and S. Vigogna, \emph{Understanding neural networks with reproducing kernel {B}anach spaces}, Appl. Comput. Harmon. Anal. \textbf{62} (2023), 194--236. \MR{4484791}

\bibitem{bogachev1998}
V.~I. Bogachev, \emph{Gaussian measures}, Mathematical Surveys and Monographs, vol.~62, American Mathematical Society, Providence, RI, 1998. \MR{1642391}

\bibitem{brzezniak1996}
Z. Brze\'zniak, \emph{On {S}obolev and {B}esov spaces regularity of {B}rownian paths}, Stochastics Stochastics Rep. \textbf{56} (1996), no.~1-2, 1--15. \MR{1396751}

\bibitem{diestel1977}
J. Diestel and J.~J. Uhl, Jr., \emph{Vector measures}, Mathematical Surveys, vol. No. 15, American Mathematical Society, Providence, RI, 1977, With a foreword by B. J. Pettis. \MR{453964}

\bibitem{driscoll1973}
M.~F. Driscoll, \emph{The reproducing kernel {H}ilbert space structure of the sample paths of a {G}aussian process}, Z. Wahrscheinlichkeitstheorie und Verw. Gebiete \textbf{26} (1973), 309--316. \MR{370723}

\bibitem{fortet1973}
R. Fortet, \emph{Espaces \`a{} noyau reproduisant et lois de probabilit\'es des fonctions al\'eatoires}, Ann. Inst. H. Poincar\'e{} Sect. B (N.S.) \textbf{9} (1973), 41--58. \MR{345193}

\bibitem{fortet1974}
\bysame, \emph{Espaces \`a{} noyau reproduisant et lois de probabilit\'e{} de fonctions al\'eatoires}, C. R. Acad. Sci. Paris S\'er. A \textbf{278} (1974), 1439--1440. \MR{370724}

\bibitem{hajek1962}
J. H\'ajek, \emph{On linear statistical problems in stochastic processes}, Czechoslovak Math. J. \textbf{12(87)} (1962), 404--444. \MR{152090}

\bibitem{hytonen2017}
T. Hyt\"onen, J. van Neerven, M. Veraar, and L. Weis, \emph{Analysis in {B}anach spaces. {V}ol. {II}. {P}robabilistic methods and operator theory}, Ergebnisse der Mathematik und ihrer Grenzgebiete. 3. Folge. A Series of Modern Surveys in Mathematics [Results in Mathematics and Related Areas. 3rd Series. A Series of Modern Surveys in Mathematics], vol.~67, Springer, Cham, 2017. \MR{3752640}

\bibitem{kallianpur1970a}
G.~Kallianpur, \emph{The role of reproducing kernel {H}ilbert spaces in the study of {G}aussian processes}, Advances in {P}robability and {R}elated {T}opics, {V}ol. 2, Dekker, New York, 1970, pp.~49--83. \MR{283866}

\bibitem{kallianpur1970b}
\bysame, \emph{Zero-one laws for {G}aussian processes}, Trans. Amer. Math. Soc. \textbf{149} (1970), 199--211. \MR{266293}

\bibitem{lepage1973}
R. Le~Page, \emph{Subgroups of paths and reproducing kernels}, Ann. Probability \textbf{1} (1973), 345--347. \MR{350835}

\bibitem{lukic2001}
M.~N. Luki\'c and J.~H. Beder, \emph{Stochastic processes with sample paths in reproducing kernel {H}ilbert spaces}, Trans. Amer. Math. Soc. \textbf{353} (2001), no.~10, 3945--3969. \MR{1837215}

\bibitem{ostrovskii2001}
M.~I. Ostrovskii, \emph{Weak{$^*$} sequential closures in {B}anach space theory and their applications}, General topology in {B}anach spaces, Nova Sci. Publ., Huntington, NY, 2001, pp.~21--34. \MR{1901532}

\bibitem{parzen1963}
E. Parzen, \emph{Probability density functionals and reproducing kernel {H}ilbert spaces}, Proc. {S}ympos. {T}ime {S}eries {A}nalysis ({B}rown {U}niv., 1962), Wiley, New York-London, 1963, pp.~155--169. \MR{149634}

\bibitem{pettis1938}
B.~J. Pettis, \emph{On integration in vector spaces}, Trans. Amer. Math. Soc. \textbf{44} (1938), no.~2, 277--304. \MR{1501970}

\bibitem{rudin1991}
W. Rudin, \emph{Functional analysis}, second ed., International Series in Pure and Applied Mathematics, McGraw-Hill, Inc., New York, 1991. \MR{1157815}

\bibitem{vakhania1987}
N.~N. Vakhania, V.~I. Tarieladze, and S.~A. Chobanyan, \emph{Probability distributions on {B}anach spaces}, Mathematics and its Applications (Soviet Series), vol.~14, D. Reidel Publishing Co., Dordrecht, 1987, Translated from the Russian and with a preface by Wojbor A. Woyczynski. \MR{1435288}

\bibitem{neerven2010}
J. van~Neerven, \emph{{$\gamma$}-radonifying operators---a survey}, The {AMSI}-{ANU} {W}orkshop on {S}pectral {T}heory and {H}armonic {A}nalysis, Proc. Centre Math. Appl. Austral. Nat. Univ., vol.~44, Austral. Nat. Univ., Canberra, 2010, pp.~1--61. \MR{2655391}

\end{thebibliography}
\end{document}